\documentclass[11pt,leqno,a4paper]{amsart}
\usepackage{amsmath}
\usepackage{amsthm}
\usepackage{amssymb}
\usepackage{amscd}
\usepackage{mathabx}
\usepackage{accents}

\usepackage{tikz}
\usetikzlibrary{matrix,arrows,decorations.pathmorphing}

\usepackage{enumerate}
\usepackage{verbatim}

\usepackage{hyperref}
\usepackage{mathabx}

\usepackage[english]{babel}
\usepackage[utf8]{inputenc}
\usepackage[totalheight=22 true cm, totalwidth=14 true cm]{geometry}
\usepackage{mathrsfs}
\usepackage[all]{xy}

\newtheorem{thm}{Theorem}[section]
\newtheorem{cor}[thm]{Corollary}
\newtheorem{lemma}[thm]{Lemma}

\newtheorem{prop}[thm]{Proposition}

\newtheorem{defn}[thm]{Definition}
\newtheorem{conj}[thm]{Conjecture}

\newtheorem{open}[thm]{Open question}

\theoremstyle{remark}

\theoremstyle{definition}

\newtheorem{rmk}[thm]{Remark}
\newtheorem{exa}[thm]{Example}

\numberwithin{equation}{thm}

\def\beq{\begin{equation}}
\def\eeq{\end{equation}}

\def\beqn{\begin{equation*}}
\def\eeqn{\end{equation*}}

\def\ben{\begin{enumerate}}
\def\een{\end{enumerate}}

\def\crash#1{}
\def\A{{\mathbb A}}
\def\B{{\mathbb B}}
\def\C{{\mathbb C}}

\def\N{{\mathbb N}}

\def\R{{\mathbb R}}

\def\ie{\emph{i.e.}~}

\def\cI{{\mathcal I}}
\def\cM{{\mathcal M}}

\def\cO{{\mathcal O}}

\def\cW{{\mathcal W}}

\def\sF{{\mathscr F}}

\def\sI{{\mathscr I}}

\def\sP{{\mathscr P}}

\def\bC{{\mathbf C}}
\def\bD{{\mathbf D}}

\def\bI{{\mathbf I}}

\def\bU{{\mathbf U}}

\def\a{\alpha}

\def\la{\lambda}

\def\ol{\overline}

\def\limpro{\mathop{\lim\limits_{\displaystyle\leftarrow}}}

\def\limind{\mathop{\lim\limits_{\displaystyle\rightarrow}}}

\def\lt{\langle}
\def\gt{\rangle}

\def\void{{\rm \varnothing}}

\def\rhook{{\hookrightarrow}} 

\def\bSet{\mathbf{Set}}

\def\op{{\rm op\,}}
\def\Roos{{\rm Roos\,}}

\def\bPro{\mathbf{Pro}}

\def\bBorn{\mathbf{Born}}
\def\bCBorn{\mathbf{CBorn}}
\def\bLoc{\mathbf{Loc}}
\def\bCoh{\mathbf{Coh}}
\def\bAb{\mathbf{Ab}}
\def\bCpt{\mathbf{Cpt}}
\def\wotimes{\widehat{\otimes}}

\author{Federico Bambozzi}

\title{Theorems A and B for dagger quasi-Stein spaces}

\thanks{
	The author acknowledges the support of the University of Padova by MIUR PRIN2010-11 ``Arithmetic Algebraic Geometry and Number Theory", and the University of Regensburg with the support of the DFG funded CRC 1085 ``Higher Invariants. Interactions between Arithmetic Geometry and Global Analysis" that permitted him to work on this project.}

\begin{document}

\begin{abstract}
	 In this article we use the homological methods of the theory of quasi-abelian categories and some results from functional analysis to prove Theorems A and B for (a broad sub-class of) dagger quasi-Stein spaces. In particular we show how to deduce these theorems from the vanishing, under certain hypothesis, of the higher derived functors of the projective limit functor. Our strategy of the proof generalizes and puts in a more formal framework the Kiehl's proof for rigid quasi-Stein spaces.
\end{abstract}

\maketitle

\tableofcontents

\section{Introduction}

Quasi-Stein spaces in non-Archimedean geometry were introduced by Kiehl in \cite{KI}, where he proved, with a brief and elegant argument, that the Theorems A and B hold for them (see the beginning of section \ref{sec:thm_B} for a general statement of these theorems). In \cite{GK} Grosse-Kl\"onne has introduced the quasi-Stein spaces in the context of his theory of dagger spaces. 

If a quasi-Stein space is without borders it is called Stein (important disclaimer: the word Stein in this work is used with a different meaning with respect to the use which is done in \cite{Liu}). Non-Archimedean Stein spaces are the precise analogous of their complex analytic counter-part and they have been studied by several authors: \cite{Lut}, \cite{Put} \cite{Chi}, see also Chapter 6 of \cite{Poi} for a discussion about Stein spaces in the context of global analytic spaces. 
It is easy to check that the notion of Stein space given in classical rigid geometry by Kiehl is equivalent to the notion given by dagger analytic geometry. Therefore, Kiehl's Theorem applies also for dagger Stein spaces, which canonically coincide with the non-dagger ones. But for general dagger quasi-Stein spaces the question of whether they satisfy Theorems A and B turned out to be harder to settle.

In this work, we reduce the problem of proving Theorem B (and as a consequence Theorem A) for dagger quasi-Stein spaces to the calculation of the derived functors of the projective limit functor. In particular we use the theory of quasi-abelian categories, as developed by Schneiders \cite{ScQA} and later Prosmans \cite{Pr}, to define in a rigorous way the derived functors of the projective limit functor for locally convex and bornological vector spaces. We will then use, adapt and further strength, some advanced results in functional analysis that permits to check, under suitable circumstances, that the projective limit functor is an exact functor (see Definition \ref{defn:exact_fucntor} for the definition of exact functor in the context of quasi-abelian categories). This reduction step permits to overcome the difficulties thanks to the literature available in the hard topic of studying projective limits of locally convex spaces. 

In more detail, the paper is structured in the following way. In section 2 we introduce the notation and briefly review the main concepts that are used in the paper. Section 3 starts studying the derived functors of the projective limit functor in a general quasi-abelian category. After that, the results will be specialized for the case $\bLoc_k$ and $\bBorn_k$, the categories of locally convex and bornological vector spaces of convex type (we are motivated to consider also bornological vector spaces for the relevance of these with our previous works: \cite{Bam}, \cite{BaBe}, \cite{BaBeKr}; we say more on this in the concluding section). We give conditions under which the projective limit of a countable family of LB spaces has vanishing higher derived functors, in the sense of the theory of quasi-abelian categories. In section 4 we use the results of section 3 to prove our main results. We begin section 4 by reviewing the notions of quasi-Stein space, coherent sheaf and the statement of the Theorems A and B. Then, we explain how coherent sheaves on analytic spaces can be considered as sheaves with values in the category $\bLoc_k$ or $\bBorn_k$. We proceed re-proving Kiehl's Theorem as a first explanation of the use of the machinery developed so far. We end section 4 proving our main result, which is Theorem \ref{thm:B_dagger}, whose immediate Corollary \ref{cor:B_dagger} is the Theorem B for dagger quasi-Stein spaces that admits a closed embedding in a direct product of affine spaces and closed polydisks. This level of generality is enough for the applications.

In the concluding section we briefly mention a couple of applications of our main result. In these applications we see how the Theorem B for dagger quasi-Stein spaces permits to simplify proofs from \cite{Bert} and \cite{DLZ} and to give to their computation a more theoretical explanation. We end the paper describing Conjecture \ref{conj:embedding} about the geometry of quasi-Stein spaces which, if proved, would imply that the main result presented in this paper in fact implies the Theorem B for all dagger quasi-Stein spaces, in full generality. Conjecture \ref{conj:embedding} is our generalization of the so-called Embedding Theorems for Stein spaces: see \S V.1 of \cite{GR} for the complex case and \cite{Lut} for the non-Archimedean version of the Embedding Theorem.

\paragraph{\textbf{Acknowledgements}}
 
The author would like to thank prof. Elmar Grosse-Kl\"onne for suggesting to him to try to prove Theorem B for dagger quasi-Stein spaces and for helpful discussions. He is thankful to Florent Martin, Lorenzo Fantini and Daniele Turchetti too, for discussions related to the topics of this article.

\section{Notation and terminology} 

We list here the notation we will use in the rest of the paper.

\begin{itemize}
	\item $k$ will always denote a fixed valued base field. By valued field we mean a field which is complete with respect to a non-trivial valuation which can be either Archimedean or non-Archimedean.
	\item $k^\times = \{ x \in k | x \ne 0 \}$, $k^\circ = \{ x \in k | |x| \le 1 \}$, $k^{\circ \circ} = \{ x \in k | |x| < 1 \}$ and $(k^\circ)^\times = k^\circ \cap k^\times$, $(k^{\circ \circ})^\times = k^{\circ \circ} \cap k^\times$.
	\item If $V$ is a $k$-vector space, by a \emph{Banach disk} of $V$ we mean a subset $B \subset V$ that is a disk (\ie it is absolutely convex) and such that the vector subspace of $V$ spanned by $B$, usually denoted $V_B$, equipped with the gauge semi-norm induced by $B$, is a Banach space.
	\item For a category $\bC$ we often use the notation $X \in \bC$ to mean that $X$ is an object of $\bC$.
	\item If $I$ is a small category, the notation $\bC^I$ will denote the category of covariant functors from $I$ to $\bC$ and the notation $\bC^{I^{op}}$ the category of contravariant functors from $I$ to $\bC$.
	\item If $\bC$ is a quasi-abelian category, then $D(\bC)$ denotes the derived category of unbounded complexes of $\bC$, $D^{\ge 0}(\bC)$ (resp. $D^{+}(\bC)$) the derived categories of complexes which are zero in degree $< 0$ (resp. bounded below), with respect to the left t-structure, and $D^{\le 0}(\bC)$ (resp. $D^{-}(\bC)$) the derived categories of complexes which are zero in degree $> 0$ (resp. bounded above).
	\item For any category $\bC$ we denote by $\bPro(\bC)$ the category of pro-objects of $\bC$.
	\item We will use the notation ${``\underset{i\in I}\limpro"} V_i$ for objects of $\bPro(\bC)$, with $V_i \in \bC$.
	\item Given two polyradii $\rho, \rho' \in \R_+^n$ with the notation $\rho < \rho'$ we mean that each component of $\rho$ is strictly smaller than the correspondent component of $\rho'$.
\end{itemize}

\paragraph{\textbf{Quasi-abelian categories}}

In this section $\bC$ will always denote a quasi-abelian category.
We recall the notions from the theory of quasi-abelian categories that we will use later on. The main reference for the theory of quasi-abelian categories is \cite{ScQA}. We do not recall here all the basics of the theory for which we refer to section 1.1 of ibid. (in particular see Definition 1.1.3 of ibid. for the definition of quasi-abelian category). One of the main properties of a quasi-abelian category $\bC$ is that the family of strictly exact sequences of $\bC$ endow $\bC$ with a structure of Quillen exact category.

\begin{defn} \label{defn:exact_fucntor}
	Let $F: \bC \to \bD$ be an additive functor between two quasi-abelian categories. $F$ is said to be \emph{exact} if it maps short strictly exact sequences of $\bC$ to short strictly exact sequences of $\bD$. $F: \bC \to \bD$ is said to be \emph{left exact} if, given a short strictly exact sequence
	\[ 0 \to X \to Y \to Z \to 0 \]
	of $\bC$, then the sequence
	\[ 0 \to F(X) \to F(Y) \to F(Z) \]
	is strictly exact.
\end{defn}

\begin{defn}
	Let $F: \bC \to \bD$ be a left exact functor. The \emph{right derived functor} of $F$, if it exists, is the functor $\R F$ which render the diagram
	\[
\begin{tikzpicture}
\matrix(m)[matrix of math nodes,
row sep=2.6em, column sep=2.8em,
text height=1.5ex, text depth=0.25ex]
{  \bC  & \bD  \\
   D^+(\bC) & D^+(\bD)   \\};
\path[->,font=\scriptsize]
(m-1-1) edge node[auto] {$$} (m-1-2);
\path[->,font=\scriptsize]
(m-1-1) edge node[auto] {$$} (m-2-1);
\path[->,font=\scriptsize]
(m-1-2) edge node[auto] {$$}  (m-2-2);
\path[->,font=\scriptsize]
(m-2-1) edge node[auto] {$$}  (m-2-2);
\end{tikzpicture}
\]
 commutative in a universal way (cf. Definition 1.3.1 of \cite{ScQA} for a fully detailed definition).
\end{defn}

\begin{defn}
Let $F: \bC \to \bD$ be an additive functor between two quasi-abelian categories. A full additive subcategory $\bI$ of $\bC$ is called \emph{$F$-injective} if
\begin{itemize}
\item for any object $X$ of $\bC$ there is an object $I \in \bU$ and a strict monomorphism $X \to I$.
\item for any strictly exact sequence
\[ 0 \to X \to X' \to X'' \to 0 \]
of $\bC$ where $X$ and $X'$ are object of $\bI$, $X''$ is also in $\bI$.
\item for any strictly exact sequence
\[ 0 \to X \to X' \to X'' \to 0 \]
where $X$, $X'$ and $X''$ are objects of $\bI$, the sequence
\[ 0 \to F(X) \to F(X') \to F(X'') \to 0 \]
is strictly exact in $\bD$.
\end{itemize}
\end{defn}

\begin{defn}
We say that $\bC$ \emph{has exact products} if for any index set $I$ the functor $\prod_{I}: \bC^{I^{op}} \to \bC$ is well defined and it is an exact functor.
\end{defn}

The heart of the left t-structure of $D(\bC)$ is an abelian category which can be interpreted as the abelian envelop of $\bC$ (cf. Proposition 1.2.36 of \cite{ScQA} for more details). Since we consider on $D(\bC)$ the left t-structure, we call its heart the \emph{left heart} of $D(\bC)$ and we denote it by $LH(\bC)$. One can check that the derived category of $LH(\bC)$ (in the usual sense since it is an abelian category) is equivalent to $D(\bC)$ and therefore one can think to complexes of $\bC$ as taking cohomology with values in $LH(\bC)$. Thus, one can define the cohomology functors $LH^n: D(\bC) \to LH(\bC)$ (cf. Definition 1.2.8 of \cite{ScQA}). In particular we have the following remark.

\begin{rmk} \label{rmk:LH_0}
A complex of objects of $\bC$, $V_\bullet$, is strictly exact if and only if $LH^n(V_\bullet) \cong 0$ for all $n$ (cf. Corollary 1.2.20 of \cite{ScQA}). 
\end{rmk}

\paragraph{\textbf{Functional analysis}} Here we recall some basic facts, notations and definitions from functional analysis that will be used throughout the paper.

\begin{defn}
We denote by $\bLoc_k$ the category of locally convex spaces over $k$.
\end{defn}

The theory of locally convex spaces is well established nowadays. Basic text books for the Archimedean part of the theory are so many that we cite only the one to which we will refer to later, \cite{J} . For the non-Archimedean theory we mention \cite{PGS} and \cite{Roo}.

\begin{defn}
We denote by $\bBorn_k$ the category of bornological vector spaces of convex type over $k$ and with $\bCBorn_k$ the category of complete bornological vector spaces over $k$.
\end{defn}

The theory of bornological vector spaces is not a mainstream topic of functional analysis as it is the theory of locally convex spaces, although it has been quite popular in the '70s. Also for bornological spaces the Archimedean side of the theory is much more developed, but our main reference \cite{H2} has been written to encompass uniformly any base field, in the same spirit of the present paper.

\begin{defn} \label{defn:t_b_functors}
There is an adjoint pair of functors ${}^t: \bBorn_k \to \bLoc_k$, ${}^b : \bLoc_k \to \bBorn_k$ defined in the following way. ${}^b$ associates to each locally convex space $E$ the bornological vector space defined on $E$ by its bounded subsets (this bornology is usually called canonical or von Neumann); ${}^t$ associates to each bornological vector space $E$ the locally convex space defined on $E$ by its bornivorous subsets (\ie the subsets that absorb all bounded subsets). Details about these constructions can be found in the Chapter 1 of \cite{H2}.
\end{defn}

We now recall the terminology about the spaces we will use in this paper. Since we will use both the category $\bLoc_k$ and $\bBorn_k$ we have two versions of the same definition in the two contexts. But the reader must be warned that this duplication process is not a futile exercise, and some basic properties of the spaces may change when considering a space as a bornological space or as a locally convex space (the most important difference probably is the failure of the open mapping theorem for Fr\'echet spaces when they are considered as bornological vector spaces). In section \ref{sec:thm_B} we will see how these different properties will affect our results depending on the choice of working on $\bLoc_k$ or $\bBorn_k$.

\begin{defn}
A \emph{Fr\'echet space} in $\bLoc_k$ is a locally convex space whose topology is metrizable and complete. A \emph{bornological Fr\'echet space}, is an object of $\bBorn_k$ which is isomorphic to $E^b$ for some Fr\'echet space $E \in \bLoc_k$.
\end{defn}

\begin{defn}
In both $\bLoc_k$ and $\bBorn_k$ an \emph{LB space} is a space which is isomorphic to a countable direct limit of Banach spaces.
\end{defn}


\begin{defn} \label{defn:nuclear_born}
A bornological space is called \emph{nuclear} if it is isomorphic to a direct limit of Banach spaces whose system maps are nuclear maps.
\end{defn}

We end this section introducing a technical notion needed to explain and prove the main results of functional analysis on which this paper relies. This notion is the notion of \emph{webbed space}. We describe (and we will use) only the topological version of this concept. Moreover, as already did by the author in \cite{Bam2}, we adapt the classical definition to encompass the case when $k$ is a non-Archimedean base field, which is usually not considered in the literature about these advanced topic in functional analysis, as for example in \cite{Wen} where the results are only stated over $\R$ and $\C$. There are several variations of the notion of webbed space, all strictly related and essentially equivalent for applications. We give the following one.

\begin{defn} \label{defn:top_webs}
A \emph{web} on a Hausdorff locally convex space $E$ is a map $\cW: \underset{i \in \N}\bigcup \N^i \to \sP(E)$ such that
\begin{enumerate}
\item the image of $\cW$ consists of absolutely convex subsets;
\item $\cW(\void) = E$;
\item given a finite sequence $(n_0, \ldots, n_i)$, then 
\[ \cW(n_0, \ldots , n_i) = \bigcup_{n \in \N} \cW(n_0, \ldots, n_i, n). \]
\item for every sequence $s: \N \to \N$, $\exists (\la_i)_{i \in \N} \in (0, 1)^{\N}$ such that for all sequences $(x_i)_{i \in \N}$ with $x_i \in \cW(s(0), \ldots , s(i))$ the series
\[ \sum_{i = 1}^\infty \mu_i x_i \]
converges in $E$ when $\mu_i \in k$ with $|\mu_i| \le \la_i$ for all $i \in \N$.
\end{enumerate}
A Hausdorff locally convex space $E$ that carries a topological web is called
\emph{webbed locally convex space}.
\end{defn}

\begin{defn}
A web $\cW$ on $E$ is said
\begin{itemize}
\item \emph{ordered} if given two sequence $s: \N \to \N$, $s': \N \to \N$ such that $s \le s'$ then $\cW(s(0), \dots, s(i)) \subset \cW(s'(0), \dots, s'(i))$ for each $i \in \N$;
\item \emph{strict} if the scalars $(\mu_i)_{i \in \N}$ in (4) of Definition \ref{defn:top_webs} can be chosen such that for all $j \in \N$ the series
\[ \sum_{i = j}^\infty \mu_i x_i \]
converges to and element of $\cW(s(0), \dots, s(j))$.
\end{itemize}
\end{defn}


\paragraph{\textbf{Analytic geometry}}

Here we establish the notation we will use for the algebras used in analytic geometry. Our approach to dagger analytic geometry will be akin to the one discussed in \cite{Bam}.

\begin{defn}
	The algebras of convergent power-series on closed polydisks of polyradius $\rho = (\rho_1, \dots, \rho_n)$ are denoted by 
	\[ T_k^n(\rho) \doteq \begin{cases}
	 \{ \underset{i \in \N^n}\sum  a_i X_i \in k \ldbrack X \rdbrack | \underset{i \in \N^n}\sum |a_i| \rho^i < \infty  \}  \ \ \text{if $k$ is Archimedean} \\
	 \{  \underset{i \in \N^n}\sum a_i X_i \in k \ldbrack X \rdbrack | |a_i| \rho^i \to 0, \text{ for } |i| \to \infty  \} \ \ \text{if $k$ is non-Archimedean }
	\end{cases},
	 \] 
	 where $X = (X_i)$ is a multi-variable and $i$ is a multi-index. $T_k^n(\rho)$ will be considered as a Banach algebra equipped with the norm
	\[ \begin{cases}
	  \underset{i \in \N^n}\sum |a_i| \rho^i < \infty   \ \ \text{if $k$ is Archimedean} \\
	 \underset{i \in \N^n} \max \{ |a_i| \rho_i \} \ \ \text{if $k$ is non-Archimedean }
	\end{cases}.
	 \] 
\end{defn}

If $k$ is non-Archimedean it is meaningful to develop a theory of affinoid algebras and affinoid spaces. We recall here the few concepts of that theory we will mainly use.

\begin{defn}
For a non-Archimedean $k$, we say that a topological or bornological $k$-algebra $A$ is an \emph{affinoid algebra} if
\[ A \cong \frac{T_k^n(\rho)}{I}, \]
for some $n \in \N$, some polyradius $\rho$ and an ideal $I \subset T_k^n(\rho)$, where on $\frac{T_k^n(\rho)}{I}$ is considered the quotient norm induced by $T_k^n(\rho)$ with its associated topology or canonical bornology.
\end{defn}

The dual category of the category of affinoid algebras is equivalent to the category of affinoid spaces. These affinoid spaces are used as building blocks of non-Archimedean analytic geometry. The affinoid space associated to the affinoid algebra $A$ will be denoted with $\cM(A)$ (implicitly referring to its representation by the Berkovich spectrum). 

\begin{defn}
Let $A$ be an affinoid $k$-algebra. One can define the algebra
\[ A \lt \rho^{-1} X \gt = \{ \underset{i \in \N^n} \sum a_i X_i \in A \ldbrack X \rdbrack | |a_i| \rho^i \to 0 \} \]
and for each $f_1, \dots, f_n \in A$ we can consider the morphism
\[ A \to \frac{A \lt \rho^{-1} X \gt}{(X_1 - f_1, \dots, X_n - f_n)} = A \lt f_1, \dots, f_n \gt. \]
We call it a \emph{Weierstrass localization}. The dual map of a Weierstrass localization $\cM(A \lt f_1, \dots, f_n \gt) \to \cM(A)$ is called a \emph{Weierstrass subdomain embedding} and it defines an admissible open subset of $\cM(A)$ with respect to the so-called weak G-topology.
\end{defn}

Without any restriction on $k$ one can develop the theory of dagger affinoid algebras, as did in \cite{Bam}. We recall here the basic definitions.

\begin{defn}
The algebras of overconvergent power-series on closed polydisks of polyradius $\rho = (\rho_1, \dots, \rho_n)$ are denoted by 
\[ W_k^n(\rho) = \limind_{\rho' > \rho } T_k^n(\rho'). \]
On $W_k^n(\rho)$ we consider the direct limit (LB) bornology.
\end{defn}

\begin{defn}
We say that a bornological $k$-algebra $A$ is a \emph{dagger affinoid algebra} if
\[ A \cong \frac{W_k^n(\rho)}{I}, \]
for some $n \in \N$, some polyradius $\rho$ and an ideal $I \subset W_k^n(\rho)$, where on $\frac{W_k^n(\rho)}{I}$ is considered the quotient bornology induced by $W_k^n(\rho)$.
\end{defn}

\begin{rmk}
In the previous work of the author, \cite{Bam}, \cite{Bam2}, \cite{BaBe}, \cite{BaBeKr}, dagger affinoid algebras were always considered as bornological algebras. We refer to those works for an explanation of why it is usually better to consider them as bornological algebras instead of topological algebras. On the contrary, in this work, we also consider on them the canonical LB topology, because we would like, for the sake of completeness, to prove our vanishing results on the cohomology in both cases: when coherent sheaves over a dagger affinoid space are considered are sheaves of bornological spaces and when they are considered as sheaves of locally convex spaces. It is also a more natural way to work out our proof and link it with existing results in literature. The canonical LB topology on dagger affinoid algebra can be obtained applying the functor ${}^t$ described in Definition \ref{defn:t_b_functors}. Notice also that for any dagger affinoid algebra $A$ one has that $A \cong ((A)^t)^b$.
\end{rmk}

Again, the dual category of the category of dagger affinoid algebras is equivalent to the category of dagger affinoid spaces. Dagger affinoid spaces are the building blocks of dagger analytic geometry and the dagger affinoid space associated with to the affinoid algebra $A$ will be denoted with $\cM(A)$. Finally, also for dagger affinoid algebras we can introduce Weierstrass localizations.

\begin{defn}
Let $A$ be a dagger affinoid $k$-algebra. One can define the algebra
\[ A \lt \rho^{-1} X \gt^\dagger = W_k^n(\rho) \wotimes_k A \]
and for each $f_1, \dots, f_n \in A$ we can consider the morphism
\[ A \to \frac{A \lt \rho^{-1} X \gt^\dagger}{(X_1 - f_1, \dots, X_n - f_n)} = A \lt f_1, \dots, f_n \gt^\dagger. \]
We call it a \emph{(dagger) Weierstrass localization}. The dual map of a (dagger) Weierstrass localization $\cM(A \lt f_1, \dots, f_n \gt) \to \cM(A)$ is called a \emph{(dagger) Weierstrass subdomain embedding} and it defines an admissible open subset of $\cM(A)$ with respect to the weak G-topology.
\end{defn}

\section{Derived functors of the projective limit functor} \label{sec:der_lim}

In this section we recall the main facts about the main tools we use in our proof of Theorem B for dagger quasi-Stein spaces: the derived functors of the projective limit functor in the context of quasi-abelian categories. In the first section we recall the general theory about the derivation of the projective limit functors of quasi-abelian categories from \cite{Pr}. Then, we discuss criteria for the vanishing of the higher derived functors in the category of locally convex spaces and in the category of bornological vector spaces.

\subsection{Projective limits in quasi-abelian categories}

In this section $\bC$ will denote a quasi-abelian category. Here we mainly recall results and definitions from \cite{Pr}. 

\begin{defn}
Let $I$ be a small category and $F: I^{\op} \to \bC$ an additive functor. The \emph{Roos complex} of $F$ is defined to be the complex in degree $\ge 0$ given by
\[ Roos(F) \doteq (\Roos^n(F) \doteq \prod_{i_0 \to i_1 \to \dots \to i_n} F(i_0))_{n \in \N} \]
where $i_0 \to i_1 \to \dots \to i_n$ is a chain of composable morphisms of $I$. We refer the reader to the Definition 3.2.1 of \cite{Pr} for the (quite long) definition of the differentials of this complex, since it is not crucial for our exposition.
\end{defn}

\begin{defn}
An object $F \in \bC^{I^{\op}}$ is called \emph{Roos acyclic} if its Roos complex is strictly quasi-isomorphic to a complex concentrated in degree $0$.
\end{defn}

From now on we will always suppose that $\bC$ has exact products.

\begin{prop} \label{prop:inv_prosmans}
	Let $I$ be a small category and let $\bC$ has exact products. Then, the family of objects in $\bC^{I^{op}}$ which are Roos-acyclic form a $\underset{i\in I}\limpro$-injective subcategory. In particular, the functor 
	\[ \limpro_{i\in I}:\bC^{I^{op}} \to \bC \]
	is right derivable to a functor 
	\[ D^{\ge 0}(\bC^{I^{op}} ) \to D^{\ge 0}(\bC ) \]
	and for any object $V_\bullet \in \bC^{I^{op}}$, we have a canonical isomorphism 
	\begin{equation}\label{eqn:Derived2Roos}
	\R \limpro_{i\in I} V_{i}\cong \Roos(V)
	\end{equation}
	where the right hand side is the Roos complex of $V$.
\end{prop}
{\bf Proof.} 
See Proposition 3.3.3 of \cite{Pr}.
\hfill $\Box$

In the case when $I = \N$ the Roos complex is particularly simple to calculate.

\begin{lemma}\label{lem:RoosML}
	Let $\bC$ be a quasi-abelian category with exact products. Then, for a functor $V: \N^{op} \to \bC$ the Roos complex is strictly isomorphic to the complex 
	\begin{equation} \label{eqn:Weibel} 
	0 \to \limpro_{i \in \mathbb{N}}V_{i} \to \prod_{i=0}^{\infty} V_{i} \stackrel{\Delta}\longrightarrow \prod_{i=0}^{\infty} V_{i} \to 0
	\end{equation}
	where 	
\[ \Delta = (\dots, id_{V_2} - \pi_{2, 3}, id_{V_1} - \pi_{1, 2}, id_{V_0} - \pi_{0, 1})
\]
	and $\pi_{i, i+1}: V_{i + 1} \to V_i$ denote the system morphisms of $V$.
\end{lemma}
{\bf Proof.}
Since the cardinality of the set of natural numbers is less than the second infinite cardinal, Proposition 5.2.3 of \cite{Pr} implies that $LH^{n}(\R \underset{i\in I}\limpro V_{i})=0$ for all $n\geq 2$. The fact that $\R \underset{i\in I}\limpro V_{i}$ is represented by (\ref{eqn:Weibel}) follows from the general definition of the Roos complex, specialized with $I = \N$, and the isomorphism in (\ref{eqn:Derived2Roos}).  
\hfill $\Box$

\begin{cor} \label{cor:Weibel}
Let $\{ V_{i} \}_{i \in \N}$ be a projective system in $\bC$, then  $\R \underset{i\in \N}\limpro V_{i} \cong \underset{i\in \N} \limpro V_{i}$ if and only if the $\Delta$ map of the complex (\ref{eqn:Weibel}) is a strict epimorphism.
\end{cor}
\begin{proof}
Immediate consequence of Lemma \ref{lem:RoosML}.
\end{proof}

For the sake of clarity we put the following easy lemma, to be used later on.

\begin{lemma} \label{lemma:complexes}
Let $(V_\bullet^n)_{n \in \N}$ be a complex of elements of $\bC^{I^\op}$, such that for each $n \in \N$ the functor $\underset{i \in I}\limpro V_i^n: I^\op \to \bC$ is Roos acyclic. Then,
\[ L H^n(\limpro_{i \in I} V_i^\bullet) \cong \limpro_{i \in I} (L H^n (V_i^\bullet)) \]
is an isomorphism for each $n \in \N$.
\end{lemma}
\begin{proof}
It follows easily by splitting the complex $\underset{i \in I}\limpro V_i^\bullet$ in strictly short exact sequences and apply Lemma 3.81 of \cite{BaBeKr}.
\end{proof}

\begin{lemma} \label{lemma:pro_object}
	Let $\{ V_i \}_{i \in I}$ and $\{ W_j \}_{j \in J}$ be two projective systems such that  $\underset{i \in I}{``\limpro"} V_i \cong \underset{j \in J}{``\limpro"} W_j$ as objects of $\bPro(\bC)$. Then
	\[  \R \underset{i \in I}{\limpro} V_i \cong \R \underset{j \in J}{\limpro} W_j \]
	in $D^{\ge 0}(\bC)$.
\end{lemma}
\begin{proof}
	See Corollary 3.7.3 of \cite{Pr}.
\end{proof}

From next section we will only consider projective systems indexed by $\N$.
Therefore, we fix once for all the notation for a projective system $\{V_n \}_{n \in \N}$: With
\[ \pi_i: \limpro_{n \in \N} V_n \to V_i \]
we denote the canonical morphisms and with
\[ \pi_{i,j}: V_j \to V_i \]
we denote the system morphisms for $i,j \in \N$.

\subsection{Projective limits of locally convex and bornological vector spaces}
The main reference for the functional analytic results of this section is \cite{Wen}. Indeed, the main purpose of this section is to re-interpret the results proved there in the context of quasi-abelian category theory and to show that the proofs easily adapt to encompass also the case when $k$ is a non-Archimedean, non-trivially valued field.


\begin{prop} 
	Both $\bLoc_k$ and $\bBorn_k$ have the following properties:
	\begin{enumerate}
	\item they are quasi-abelians;
	\item they have exact products;
	\item the forgetful functors $\bLoc_k \to \bSet_k$ and $\bBorn_k \to \bSet_k$ commute with limits;
	\item limits of complete objects are complete objects;
	\item strict epimorphisms are srujective maps that induces the quotient topology or the quotient bornology.
	\end{enumerate}
\end{prop}
{\bf Proof.} 
All the claims of this proposition need easy verifications. We give some references in literature for them. Notice that the references to the literature usually consider only the case when $k = \C$ or $k = \R$, but the same proofs work for any base field.

\begin{enumerate}
\item See Proposition 1.8 of \cite{PrSc} and Proposition 2.1.11 of \cite{Pr2}.
\item See Proposition 1.9 of \cite{PrSc} and Proposition 2.2.2 of \cite{Pr2}.
\item Immediate consequence of the computation of the limits.
\item The inclusion functor of the full sub-categories of complete objects in $\bLoc_k$ and in $\bBorn_k$ have a left adjoint which is called the completion functor.
\item See Corollary 1.7 of \cite{PrSc} and Lemma 2.1.10 (iii) of \cite{Pr2}.
\end{enumerate}
\hfill $\Box$

Let us study first the projective systems of locally convex spaces. The next lemma the classical Mittag-Leffler lemma for Fr\'{e}chet spaces.

\begin{lemma} \label{lem:TopML} 
	Let $\{ V_i \}_{i \in \N}$ be a projective system of Fr\'{e}chet spaces in $\bLoc_k$ indexed, for which all system morphisms are dense. Then, the Roos complex
	\begin{equation} 
	0 \to \limpro_{i \in \mathbb{N}}V_{i} \to \prod_{i \in \mathbb{N}}V_{i} \stackrel{\Delta}\to \prod_{i \in \mathbb{N}}V_{i} \to 0
	\end{equation}
	is strictly exact in $\bLoc_k$.
\end{lemma}
{\bf Proof.} 
The lemma follows immediately from Theorem 3.2.1 of \cite{Wen} and the Open Mapping Theorem for Fr\'echet spaces.
\hfill $\Box$

We need more advanced results. As explained in Corollary \ref{cor:Weibel}, the problem of verifying that $\underset{i \in \N}\limpro V_i$ is exact is equivalent to check that the map
\[ \prod_{i \in \mathbb{N}}V_{i} \stackrel{\Delta_{V_\bullet}}\to \prod_{i \in \mathbb{N}}V_{i} \]
is a strict epimorphism, \ie a quotient map of locally convex spaces. In this case we can explicitly write
\[ \Delta_{V_\bullet}(\dots ,a_2,a_1,a_{0}) = (\dots, a_2 - \pi_{2,3}(a_3),a_1 - \pi_{1,2}(a_2),a_0 - \pi_{0,1}(a_1)). \]
For the sake of clarity, we divide the problem in two steps: one step is to check when the map $\Delta_{V_{\bullet}}$ is surjective and the other step is to check when from the surjectivity of $\Delta_{V_{\bullet}}$ we can also deduce that it is even a quotient map.
We do not enter in a study of all general results for locally convex spaces, but in principles the non-Archimedean version of all the results of \cite{Wen} should hold. We prove here only the ones that we strictly need.

The next is our version of Theorem 3.2.9 of \cite{Wen}.

\begin{thm} \label{thm:LB_proj_0}
	Let $\{ V_i \}_{i \in \N}$ be a projective system of LB spaces. Then, $\Delta_{V_\bullet}$ is surjective if and only if there is a sequence of Banach disks $B_n \subset V_n$, such that
	\[ \pi_{n, m}(B_m) \subset B_n, \ \ \forall m,n \in \N, m \ge n  \]
	and 
	\[ \forall n \in \N, \exists m \ge n, \forall l \ge m, \pi_{n, m}(V_m) \subset \pi_{n, l}(V_l) + B_n. \]
	.
\end{thm}
\begin{proof}
The adaptation of the proof of Theorem 3.2.9 of \cite{Wen} to the non-Archimedean case is very easy, because it relies on lemmas which use only the general notion of metric space, without any convexity assumption. Therefore, we can use the same proof noticing that $B_n \subset V_n$ are Banach disks which induce metric topologies on the subspaces of $V_n$ they span, which satisfies the hypothesis of Theorem 3.2.1 and Theorem 3.2.5 of \cite{Wen}, which are abstract results on metrizable groups required to prove the theorem.
\end{proof}

We also need a stronger version of last theorem. And to prove it we need two lemmas.

\begin{lemma} \label{lemma:LB_proj_0_strong_1}
Let $f: X \to Y$ be a continuous linear map between Hausdorff locally convex spaces and let $A \subset X$ and $B \subset Y$ be Banach disks. Then,
$f(A) + B$ and $A \cap f^{-1} (B)$ are Banach disks.
\end{lemma}
\begin{proof}
Consider the subspace $X_{A \cap f^{-1} (B)} \subset X$ spanned by $A \cap f^{-1} (B)$ and the subspace $Y_{f(A) + B} \subset Y$ spanned by $f(A) + B$. There is an exact sequence 
\[ 0 \to X_{A \cap f^{-1} (B)} \stackrel{i}{\to} X_A \times Y_B \stackrel{q}{\to} Y_{f(A) + B} \to 0 \]
where $i(x) = (x, -f(x))$ embeds the graph of $f|_{X_{A \cap f^{-1} (B)}}$ in $X_A \times Y_B$, where $X_A \subset X$ is the Banach subspace of $X$ spanned by $A$ and $Y_B \subset Y$ is the Banach subspace of $Y$ spanned by $B$, and $q(x,y) = f(x) + y$. Since the unit ball of $Y_{f(A) + B}$, is $f(A) + B$ and it is a sum of two bounded disks of $Y$, then $Y_{f(A) + B}$ is separated. But a separated quotient of a Banach space is a Banach space and $X_{A \cap f^{-1} (B)}$, being the kernel of $q$ is a normed space mapped onto a closed subspace of $X_A \times Y_B$. Therefore, both $f(A) + B$ and $A \cap f^{-1} (B)$ are Banach disks.
\end{proof}

\begin{lemma} \label{lemma:LB_proj_0_strong_2}
Let $X$ be a Hausdorff topological vector space, $A$ a bounded
subset of $X$, and $B \subset X$ a Banach disk. Suppose that there exists a $\la \in k$, with $|\la| = \frac{1}{2}$ such that $A \subset B + \la A$, then $A \subset \la' B$ with $|\la'| = 2 + \epsilon$, with $\epsilon > 0$.
\end{lemma}
\begin{proof}
Given $a \in A$ we choose inductively $a_n \in A$ and $b_n \in B$ such that
\[ a = b_0 + \la a_0 = b_0 + \la (b_1 + \la a_1) = \sum_{i = 0}^{n} \la^i b_i + \la^n a_n \]
The first sum converges in the vector subspace spanned by $B$ in $X$ to an element $b \in \ol{\la^{-1} B} \subset \la' B$ (where the closure is taken with respect to the gauge norm induced by $B$) for any $\la'$ with $|\la'| > 2$. Since $\la^n a_n$ tends to $0$ in $X$ by the boundedness of $A$, we get $a = b \in \la' B.$
\end{proof}

\begin{thm} \label{thm:LB_proj_0_strong}
	Let $\{ V_i \}_{i \in \N}$ be a projective system of LB spaces.
	The following two conditions are equivalent:
	\begin{enumerate}
	\item $\Delta_{V_\bullet}$ is surjective;
	\item \begin{equation} \label{eqn:statement}
	 \forall n \in \N, \exists m \ge n, B_n \subset V_n, \pi_{n, m}(V_m) \subset \pi_{n}(\limpro_{i \in \N} V_i) + B_n
	 \end{equation}
	 where $B_n$ are bounded.
	\end{enumerate}
\end{thm}
\begin{proof}
We are adapting the proof of Theorem 3.2.16 of \cite{Wen}. Suppose that condition (2) holds. We can assume $m = n + 1$ without loss of generality. Let $\{ B_{n,l} \}_{l \in \N}$ be fundamental sequences of the Banach disks of $V_n$. Notice that for each $n \in \N$ we have that $V_n = \underset{l \in \N} \bigcup B_{n,l}$ (because $B_{n,l}$ is a base for the canonical bornology of $V_n$). We have that
\[ \pi_{0, 1}(B_1) \subset \pi_{0}(\limpro_{i \in \N} V_i) + B_0 \subset \pi_{0}(\pi_2^{-1}(\bigcup_{l \in \N} V_{2, l})) + B_0 = \bigcup_{l \in \N} \pi_{0}(\pi_2^{-1}(V_{2, l})) + B_0 . \]
Since the vector subspace of $V_0$ spanned by $\pi_{0, 1}(B_1)$ is Banach, we denote it by $\lt \pi_{0, 1}(B_1) \gt$, there exists a $l_2 \in \N$ such that
\[ \lt  \pi_{0, 1}(B_1) \gt \cap (\pi_{0}(\pi_2^{-1}(V_{2, l_2})) + B_0) \]
is a non-meager subset of $\lt \pi_{0, 1}(B_1) \gt$. Repeating the reasoning we can find an $l_3 \in \N$ such that
\[ \lt  \pi_{0, 1}(B_1) \gt \cap (\pi_{0}(\pi_2^{-1}(V_{2, l_2}) \cap \pi_3^{-1}(V_{3, l_3}) ) + B_0) \]
is not meager in $\lt \pi_{0, 1}(B_1) \gt$, and therefore, by induction, for any $h \ge 2$, we can find $l_h$'s such that for a fixed $j \ge 2$
\[ \lt  \pi_{0, 1}(B_1) \gt \cap (\pi_{0}( \bigcap_{h \ge 2}^j \pi_h^{-1}(V_{h, l_h})) + B_0) \]
is not meager in $\lt \pi_{0, 1}(B_1) \gt$. The same conclusions hold for the bigger subsets
\[ \lt  \pi_{0, 1}(B_1) \gt \cap (\pi_{0, j}( \bigcap_{h \ge 2}^j \pi_{h, j}^{-1}(V_{h, l_h})) + B_0), \]
for any $j \ge k$.

For $j \ge 2$ we can define the sets
\[ A_j^{(0)} = \bigcap_{h \ge 2}^j \pi_{h, j}^{-1}(V_{h, l_h})  \]
which are Banach disks as a consequence of Lemma \ref{lemma:LB_proj_0_strong_1}. Since $\lt  \pi_{0, 1}(B_1) \gt \cap (\pi_{0, j}( A_j^{(0)}) + B_0)$ is not meager, its closure in $\lt  \pi_{0, 1}(B_1) \gt$ has an interior point which can be assumed to be $0$. Hence there is a $\mu_j \in (k^{\circ \circ})^\times$, such that
\[ \mu_j \pi_{0, 1}(B_1) \subset \ol{\lt  \pi_{0, 1}(B_1) \gt \cap (\pi_{0, j}( A_j^{(0)}) + B_0)} \subset  \pi_{0, j}( A_j^{(0)}) + B_0 + \la \mu_j \pi_{0, 1}(B_1)  \]
where $\la \in k$ and $|\la| = \frac{1}{2}$. Lemma \ref{lemma:LB_proj_0_strong_2} and Lemma \ref{lemma:LB_proj_0_strong_1} combined imply 
\[ \frac{\mu_j}{\la'} \pi_{0, 1}(B_1) \subset \pi_{0, j}( A_j^{(0)}) + B_0,  \]
for a suitable $\la' \in k^\times$ whose valuation can be bounded independently of $n$ and $j$, from which we can deduce that
\[ \mu_j^{(1)} B_1 \subset \pi_{1, j}( A_j^{(0)}) + \pi_{0, 1}^{-1}(B_0)  \]
where $\mu_j = \frac{\mu_j^{(1)}}{\la'}$. By the same reasoning we can produce a sequence of Banach disks $A_j^{(n)} \subset V_j$ and sequences of elements of $(k^{\circ \circ})^\times$, denoted $\{ \mu_j^{(n)} \}_{n,j \in \N}$, with $j \ge n + 2$, such that
\begin{enumerate}
\item $\mu_j^{(n)} B_{n+1} \subset \pi_{j, n+1}(A_j^{(n)}) + \pi_{n, n+ 1}^{-1}(B_n)$;
\item $\pi_{j, j+1}(A_{j+1}^{(n)}) \subset A_j^{(n)}$;
\item and replacing $A_j^{(n)}$ by $A_k^{(1)} + \dots + A_j^{(n)}$ if necessary (operation that does not affect the previous two points) we have that $A_j^{(n - 1)} \subset A_j^{(n)}$.
\end{enumerate}
Now we define inductively, $\tilde{B}_0 = B_0$ and 
\[ \tilde{B}_{n + 1} = \pi_{n, n +1}^{-1}(\tilde{B}_n) \cap (B_n + \pi_{n + 1, n + 2}(A_{n+2}^{(n)})) \]
Notice that $\tilde{B}_n$ is a Banach disk of $V_n$ as a consequence of Lemma \ref{lemma:LB_proj_0_strong_1}.

Proceeding by induction on $n$ we show that for all $n \in \N$ and $j \ge n + 1$ there are $\delta^n_j \in (k^{\circ})^\times$ with
\begin{equation} \label{eqn:star}
 \delta^n_j B_n \subset \tilde{B}_n + \pi_{n, j}(A_j^{(n - 1)}). 
\end{equation}
The base of the induction is clear with $n = 0$ and $\delta^n_j = 1$. Suppose we have found for some $n \in \N$ constants $\delta^n_j \in (k^{\circ})^\times$ such that (\ref{eqn:star}) holds for all $k \ge n + 1$. We set $\delta^{n + 1}_j = \frac{\mu_j^{(n)}}{\epsilon_j^n} \delta^n_j$ (where $\epsilon_k = 1$ if $k$ is non-Archimedean and $\epsilon_k = 2$ otherwise). Then, the point (1) and the induction hypothesis yield
\[ \delta^{n + 1}_j B_{n +1} \subset  \frac{\delta_j^n}{\epsilon_j^n} \pi_{j, n+1}(A_j^{(n)}) + \frac{1}{\epsilon_j^n} \pi_{n, n+ 1}^{-1}( \delta_j^n B_n) \subset  \]
\[ \subset \frac{1}{\epsilon_j^n} \pi_{j, n+1}(A_j^{(n)}) + \frac{1}{\epsilon_j^n} \pi_{n, n+ 1}^{-1}(\tilde{B}_n + \pi_{n, j}(A_j^{(n - 1)})). \]
One can easily check the relation
\[ \pi_{n, n+ 1}^{-1}(\tilde{B}_n + \pi_{n, j}(A_j^{(n - 1)})) \subset \pi_{n, n+ 1}^{-1}(\tilde{B}_n) + \pi_{n + 1, j}(A_j^{(n - 1)}) \]
which in combination with the point (3) gives
\[ \delta^{n + 1}_j B_{n +1} \subset \pi_{n, n+ 1}^{-1}(\tilde{B}_n) + \pi_{n + 1, j}(A_j^{(n)}). \]
Since $\pi_{n, n+ 1}^{-1}(\tilde{B}_n) \subset B_{n + 1}$ we have that
\[ \delta^{n + 1}_j B_{n +1} \subset \pi_{n, n+ 1}^{-1}(\tilde{B}_n) \cap (B_{n + 1} + \pi_{n + 1, j}(A_j^{(n)})) + \pi_{n + 1, j}(A_j^{(n)}) \]
and applying the definition of $\tilde{B}_{n + 1}$ and (2) we get
\[ \delta^{n + 1}_j B_{n +1} \subset \tilde{B}_{n + 1} + \pi_{n + 1, j}(A_j^{(n)})  \]
as required. To conclude the proof it is enough to check that from condition (\ref{eqn:star}) we can deduce the condition of Theorem \ref{thm:LB_proj_0} (which easily implies (\ref{eqn:statement})). Multiplying (\ref{eqn:statement}) by $\delta^n_j$ we get
\[ \pi_{n, n+1}(V_{n +1}) \subset \pi_{n}(\limpro_{i \in \N} V_i) + \delta^n_j B_n \subset \pi_{n, j}(V_j) + \tilde{B}_n + \pi_{n, j}(A_j^{(n-1)}) \subset \pi_{n, j}(V_j) + \tilde{B}_n \]
concluding the proof.
\end{proof}

Then, we need to find conditions which imply $\Delta_{V_\bullet}$ to be a quotient map. For this purpose we need to use the notion of topological web introduced so far.

\begin{lemma} \label{lemma:web_LB}
Let $X$ be an LB space, then $X$ has a strict ordered web.
\end{lemma}
\begin{proof}
	Let $\{ B_n \}_{n \in \N}$ be a base for the canonical bornology of $X$, then
	\[ \cW(s(1), \dots, s(i)) \doteq \{ \mu B_{s(1)} | \mu \in k, |\mu| = \min \{ i, s(1), \dots, s(i) \} \} \]
	is a strictly ordered web on $X$.
\end{proof}

\begin{lemma} \label{lemma:web_quotient}
Let $\{ V_i \}_{i \in \N}$ be a projective system of Hausdorff locally convex spaces. If each $X_n$ has
a strict ordered web $\cW^n$, then the surjectivity of $\Delta_{V_\bullet}$ implies that
\[ \exists s: \N \to \N, \forall n \in \N, \exists m \ge n, \forall l \ge m  \]
\[ \pi_{n,m}(V_m) \subset \pi_{n}(\limpro_{i \in \N} V_i) + \bigcap_{j = 1}^n \pi_{n,j}^{-1}(\cW^j(s(j), \dots, s(n))) \]
\end{lemma}
\begin{proof}
	For the proof of this lemma we refer to Theorem 3.2.13 of \cite{Wen}, since only results valid for any metrizable topological group are used in the proof of that result.
\end{proof}


\begin{lemma} \label{lemma:web_quotient2}
Let $\{ V_i \}_{i \in \N}$ be a projective system of locally convex spaces,
the following conditions are equivalent:
\begin{itemize}
\item $\R \underset{i \in \N} \limpro V_i \cong \underset{i \in \N} \limpro V_i$;
\item $\forall n \in \N$ and $\forall U$ open $0$-neighborhood of $V_n$, $\exists m \ge n$ such that $\pi_{n, m}(V_m) \subset \pi_n(\underset{i \in \N}\limpro V_i) + U$.
\end{itemize}
\end{lemma}
\begin{proof}
For the proof of this lemma we refer to Theorem 3.3.1 of \cite{Wen}. Also for this result the argumentations of \cite{Wen} are valid, with only small obvious changes, for any base field.
\end{proof}

\begin{thm} \label{thm:web_quotient}
	Let $\{ V_i \}_{i \in \N}$ be a projective system of locally convex spaces such that $\Delta_{V_\bullet}$ is surjective. Then, if $V_i$ have strict ordered webs $\Delta_{V_\bullet}$ is a quotient map.
\end{thm}
\begin{proof}
This proof is essentially the same proof of Theorem 3.3.3 of \cite{Wen}. Let \[ \cW^n = \{ \cW^n (s(0), \dots, s(i)) \subset V_n | s \in \N^\N, i \in \N \} \]
be strict ordered webs in $V_n$. Lemma \ref{lemma:web_quotient} implies that there is $s \in \N^\N$ such that for each $n \in \N$ there is $m \in \N$ with
\[ \pi_{n, m}(V_m) \subset \pi_n(\limpro_{i \in \N} V_i) + \bigcap_{j = 1}^n \pi_{j, n}^{-1}(\cW^j(s(j), \ldots, s(n))). \]
We fix $n \in \N$ and $U$ to be a zero neighborhood of $V_n$. Then, there are $j \ge n$ and $S > 0$ such that $\cW^j (s(j), \ldots, s(n)) \subset S U$, since otherwise there would be natural numbers $n_l < n_{l+1}$ and $x_l \in \frac{1}{\la^{l+1}} \cW( s (n), \dots,s(n_l))$ with $x_l \notin U$ for  all $l \in \N$ and $0 < |\la| < 1$, contradicting the convergence of the series $\underset{l=1}{\overset{\infty}\sum} x_l$. Then
\[ \pi_{n, m}(V_m) \subset \pi_n(\limpro_{i \in \N} V_i) + \cW^n(s(n), \ldots, s(j)) \subset \pi_n(\limpro_{n \in \N} V_n) + S U  \]
and multiplying with $S^{-1}$ we get
\[ \pi_{n, m}(V_m) \subset \pi_n(\limpro_{i \in \N} V_i) + U  \]
which permits us to apply Lemma \ref{lemma:web_quotient2} to deduce the theorem.
\end{proof}

We now turn to the study of the case of bornological vector spaces.
The following easy observation permits to use a good part of the results we discussed for locally convex spaces for calculating the derived functors of $\underset{n \in \N} \limpro$ for bornological vector spaces.

\begin{lemma} \label{lemma:born_proj}
	Let $\{ V_n \}_{n \in \N}$ be a family of bornological vector spaces. Then, the differential map of the Roos complex 
	\[ \prod_{n \in \N} V_n \stackrel{\Delta_{V_\bullet}}{\to} \prod_{n \in \N} V_n \]
	is surjective if and only if the map
	\[ \prod_{n \in \N} (V_n)^t \stackrel{\Delta_{V_\bullet}^t}{\to} \prod_{n \in \N} (V_n)^t \]
	is surjective.
\end{lemma}
\begin{proof}
	The computation of products in $\bLoc_k$ and $\bBorn_k$ show that the forgetful functors $\bLoc_k \to \bSet$ and $\bBorn_k \to \bSet$ commute with products. The claim follows form the fact that the underlying $k$-vector spaces of $(V_n)^t$ and $V_n$ are isomorphic.
\end{proof}

\begin{cor}
	Given a strictly exact sequence $0 \to E_\bullet \to F_\bullet \to G_\bullet \to 0$ of elements of $\bCBorn_k^{\N^\op}$, then the map
	\[ \limpro_{n \in \N} F_n \to \limpro_{n \in \N} G_n \]
	is surjective if and only if the map
	\[ \limpro_{n \in \N} (F_n)^t \to \limpro_{n \in \N} (G_n)^t \]
	is surjective.
\end{cor}
\begin{proof}
	Immediate consequence of Lemma \ref{lemma:born_proj} combined with Lemma \ref{lem:RoosML}.
\end{proof}

\begin{rmk} \label{rmk:frechet}
What Lemma \ref{lemma:born_proj} does not imply is that strictly exactness is preserved, which is not true in general, even in very natural examples. Indeed, for any base field, one can find a Fr\'echet space $F$ and a closed subspace $E \subset F$ such that the quotient bornology of $\frac{F}{E}$ does not coincide with the von Neumann bornology of $\frac{F}{E}$. See Remark 3.73 of \cite{BaBeKr} for more informations on how to construct such examples.
\end{rmk}

Therefore, we can deduce the following corollary.

\begin{cor} \label{cor:proj_born_top}
	Let $\{ V_n \}_{n \in \N}$ be a projective system of bornological LB spaces. The map $\Delta_{V_\bullet}$ of the Roos complex of $\{ V_n \}_{n \in \N}$ is surjective (as map of bornological vector spaces) if and only if $\{ V_n \}_{n \in \N}$ satisfies one of the equivalent conditions of Theorem \ref{thm:LB_proj_0} or Theorem \ref{thm:LB_proj_0_strong}.
\end{cor}
\begin{proof}
Immediate consequence of Lemma \ref{lemma:born_proj}.
\end{proof}

Finally, we need to find conditions to ensure that the $\Delta$ map of the Roos complex is a quotient map of bornological vector spaces. We start with the bornological version of Mittag-Leffler Lemma.

\begin{lemma} \label{lem:BornML} 
	Let $\{ V_n \}_{n \in \N}$ be a projective system of Fr\'{e}chet spaces in $\bBorn_k$, where all system morphisms are dense and $\underset{i \in \N}\limpro V_{i}$ is nuclear (in the bornological sense see Definition \ref{defn:nuclear_born}). Then, the Roos complex
	\begin{equation} 
	0 \to \limpro_{i \in \N}V_{i} \to \prod_{i \in \N}V_{i} \stackrel{\Delta}\to \prod_{i \in \N}V_{i} \to 0
	\end{equation}
	is strictly exact.
\end{lemma}
{\bf Proof.} 
It is enough to check that also $\underset{i \in \N}\prod V_{i}$ is a nuclear bornological space, and apply 3.79 from \cite{BaBeKr}.

To see that $\underset{i \in \mathbb{N}}\prod V_{i}$ is nuclear one can write it as
\[ \prod_{i \in \mathbb{N}}V_{i} \cong \limpro_{n \in \N} \prod_{i \le n} V_{i} \] 
where the maps of the projective limit on the right hand side are the canonical projections $\underset{i \le n + 1}\prod V_{i} \to \underset{i \le n}\prod V_{i}$. Defining 
\[ \tilde{V}_{i, n} \doteq \begin{cases}
 V_i, \ \ \text{if } i \le n \\
 0 \ \ \text{otherwise}
\end{cases} \]
one easily checks that 
\[ \limpro_{n \in \N} \prod_{i \le n} V_{i} \cong \limpro_{n \in \N} \prod_{i \in \N} \tilde{V}_{i, n} \cong \prod_{i \in \N} \limpro_{n \in \N} \tilde{V}_{i, n} \] 
and $\underset{n \in \N}\limpro \tilde{V}_{i, n} \cong \underset{n \in \N}\limpro V_n$ (non-canonically) is nuclear for each $i$. Therefore, applying Proposition 3.51 from \cite{BaBeKr} we obtain that $\underset{i \in \mathbb{N}} \prod V_{i}$ is nuclear.
\hfill $\Box$

\begin{rmk}
Notice that in Lemma \ref{lem:BornML} there is the additional, important, hypothesis that $\underset{i \in \N}\limpro V_{i}$ is nuclear with respect to Lemma \ref{lem:TopML}. Indeed, without this hypothesis one can find bornological Fr\'echet spaces for which the lemma does not hold (as mentioned in Remark \ref{rmk:frechet}), in contrast with the topological case of the Mittag-Leffler Lemma.
\end{rmk}

Finally, we prove our result which ensure that $\Delta_{V_{\bullet}}$ is automatically a quotient map for a projective system of bornological LB spaces.

\begin{thm} \label{thm:born_LB_quotient}
	Let $\{ V_n \}_{n \in \N}$ be a projective system of bornological LB spaces such that $\Delta_{V_{\bullet}}$ is surjective. Then, $\Delta_{V_{\bullet}}$ is a quotient map. 
\end{thm}
\begin{proof}
The bornology of $\underset{n \in \N} \prod V_n$ is given by the bornology generated by subsets of the form $\underset{n \in \N}\prod B_n$, for $B_n \subset V_n$ bounded. By the explicit definition of $\Delta_{V_\bullet}$, there is a commutative diagram
	\[
\begin{tikzpicture}
\matrix(m)[matrix of math nodes,
row sep=2.6em, column sep=2.8em,
text height=1.5ex, text depth=0.25ex]
{  \underset{n \in \N}\prod V_n  & \underset{n \in \N}\prod V_n  \\
   V_{n+ 1} \times V_n & V_n   \\};
\path[->,font=\scriptsize]
(m-1-1) edge node[auto] {$\Delta_{V_{\bullet}}$} (m-1-2);
\path[->,font=\scriptsize]
(m-1-1) edge node[auto] {$\pi'$} (m-2-1);
\path[->,font=\scriptsize]
(m-1-2) edge node[auto] {$\pi_n$}  (m-2-2);
\path[->,font=\scriptsize]
(m-2-1) edge node[auto] {$\Delta'$}  (m-2-2);
\end{tikzpicture},
\]
for each $n \in \N$. So, $\Delta'$ is also surjective (because $\pi'$ is a quotient map). Notice that $V_{n + 1} \times V_n$ has a countable base for its bornology. Therefore, we can apply Buchwalter’s Theorem (Theorem 4.9 of \cite{Bam2}) to deduce that $\Delta'$ is a quotient map. This implies that for a given $\underset{n \in \N}  \prod B_n \subset \underset{n \in \N} \prod V_n$ we can consider the pre-images $(\Delta' \circ \pi')^{-1}{B_n} = \underset{i \ne n, n + 1} \prod \{0 \} \times (\Delta')^{-1}(B_n)$ and that
\[ \bigcup_{n \in \N} (\Delta' \circ \pi')^{-1}{B_n} \]
is bounded in $\underset{n \in \N} \prod V_n$, because it is easy to verify that for each $i$
\[ \pi_i(\bigcup_{n \in \N} (\Delta' \circ \pi')^{-1}{B_n}) \subset \{0\} \cup B_i \cup \pi_i((\Delta' \circ \pi')^{-1}(B_{i -1})). \]
By construction
\[ \Delta_{V_\bullet}(\bigcup_{n \in \N} (\Delta' \circ \pi')^{-1}{B_n}) \supset \prod_{n \in \N} B_n \]
which proves that $\Delta_{V_\bullet}$ is a quotient map.
\end{proof}


\section{Theorems A and B for quasi-Stein spaces} \label{sec:thm_B}


We begin this secion by recalling the definition of quasi-Stein spaces due to Kiehl. As usual the non dagger version of teh definitions are meaningful only for non-Archimedean base fields.

\begin{defn}
	A non-Archimedean $k$-analytic space $X$ is said \emph{quasi-Stein} if it admits an affinoid covering $U_1 \subset U_2 \subset ... $ such that
	\[ X = \bigcup_{i \in \N} U_i \]
	and the restrictions maps $\cO_X(U_{i+1}) \to \cO_X(U_i)$ are Weierstrass localizations. If $U_i$ is contained in the interior of $U_{i+1}$ then $X$ is said to be \emph{Stein}.
\end{defn}

\begin{defn}
	A $k$-Fr\'echet algebra $A$ is said \emph{quasi-Stein} if it admits an isomorphism
	\[ A = \limpro_{i \in \N} A_i \]
	where $A_i$ are affinoid algebras and the maps $A_{i + 1} \to A_i$ are Weierstrass localizations. If $\cM(A_i)$ is contained in the interior of $\cM(A_{i+1})$ then $X$ is said to be \emph{Stein}.
\end{defn}

The definition of quasi-Stein space has an obvious dagger version.

\begin{defn}
	A $k$-dagger analytic space $X$ is said \emph{quasi-Stein} if it admits an affinoid covering $U_1 \subset U_2 \subset ... $ such that
	\[ X = \bigcup_{i \in \N} U_i \]
	and the restrictions maps $\cO_X(U_{i+1}) \to \cO_X(U_i)$ are Weierstrass localizations. If $U_i$ is contained in the interior of $U_{i+1}$ then $X$ is said to be \emph{Stein}.
\end{defn}

\begin{rmk}
	If $k$ is non-Archimedean the request of $\cO_X(U_{i+1}) \to \cO_X(U_i)$ to be Weierstrass localizations is equivalent to ask that the set-theoretic images of the restriction maps are bornologically dense. For $k$ Archimedean Weierstrass localizations are always dense but also non-Weierstrass localizations can have dense images.
\end{rmk}

\begin{defn}
	A $k$-bornological algebra $A$ is said \emph{quasi-Stein} if it admits an isomorphism
	\[ A = \limpro_{i \in \N} A_i \]
	where $A_i$ are dagger affinoid algebras and the maps $A_{i + 1} \to A_i$ are Weierstrass localizations. If $\cM(A_i)$ is contained in the interior of $\cM(A_{i+1})$ then $X$ is said to be \emph{Stein}.
\end{defn}


For the sake of clarity we recall the definition of coherent sheaf.

\begin{defn}
Let $(X, \cO_X)$ be a ringed space. A sheaf of $\cO_X$-modules $\sF$ is said to be coherent if 
\begin{itemize}
\item $\sF$ is of \emph{finite type}, \ie for each $x \in X$ there exists an open neighborhood $x \in U$ such that $\sF(U)$ is a finite $\cO_X(U)$-module;
\item for any open set $U \subset X$, any $n \in \N$ and any morphism $\varphi \colon \cO_X^n|_U \to \sF|_U$ of $\cO_X$-modules, the kernel of $\varphi$ is finitely generated.
\end{itemize}
We denote the category of coherent sheaves over $(X, \cO_X)$ by $\bCoh(X)$.
\end{defn}

Our aim is to prove that dagger quasi-Stein spaces satisfy the property stated in the next definition, which is a problem left open by the work of Grosse-Kl\"onne \cite{GK}.

\begin{defn}
We say that a locally ringed space $(X, \cO_X)$ \emph{satisfies the Theorem B} if for every $\sF \in \bCoh(X)$ one has that
	\[ H^i(X, \sF) = 0, \ \ \forall i > 0. \]
\end{defn}

\begin{defn}
Let $(X, \cO_X)$ be a locally ringed space. We say that a subspace $Y \subset X$ is \emph{closed} if $Y = \{ x \in X | f(x) = 0, \forall f \in \sI  \}$, where $\sI$ is a coherent ideal of $\cO_X$. The canonical inclusion morphism $(Y, \frac{\cO_X}{\cI}) \to (X, \cO_X)$ is called a \emph{closed immersion}.
\end{defn}

\begin{lemma} \label{lemma:thm_B_closed_immersion}
Let $(X, \cO_X)$ be a locally ringed space which satisfies Theorem B and $(Y, \cO_Y) \rhook (X, \cO_X)$ a closed immersion. Then, $Y$ satisfies Theorem B.
\end{lemma}
\begin{proof}
We can apply Theorem 6.1.10 of \cite{Poi} because the push-forward of a coherent sheaf by a closed immersion is a coherent sheaf (cf. Theorem 3 \S 16 of \cite{Serre}).
\end{proof}

We also notice that the statement known as Theorem A, \ie that all coherent sheaves are generated by global sections, is a formal consequence of Theorem B, as next proposition recall. Therefore, we focus on the study of Theorem B.

\begin{prop}
Let $(X, \cO_X)$ be a locally ringed space which satisfies Theorem B, then all coherent sheaves over $X$ are generated by global sections.
\end{prop}
\begin{proof}
See \cite{Poi}, Theorem 6.1.9.
\end{proof}

\subsection{Coherent sheaves on analytic spaces}

As stated in the previous section, Theorem B is a purely algebo-geometric statement. But when the base locally ringed space is an analytic space coherent sheaves are canonically equipped with a richer structure. Indeed, on the spaces of sections one can put canonically locally convex topologies and bornologies of convex type. We do not enter here in a detailed explanation of how to do it (anyway it will be clear from our proof of Theorem B). We refer the reader to \S V.6 of \cite{GR} for a description of the Fr\'echet topology on the modules of sections of a coherent sheaf over a complex analytic spaces. For the non-Archimedean case, we can refer Chapter 2 of \cite{Ber} where it is shown how to put Banach norms on spaces of sections and we refer to Chapter 6 of \cite{Bam} for the dagger analytic cases, where LB bornologies are put on the spaces of sections of a coherent sheaf over a (dagger) affinoid space. Therefore, it is meaningful to consider a coherent sheaf over an analytic space $X$ not only as a sheaf of abelian groups but also as a sheaf of locally convex or bornological vector spaces.

Considering this additional structure broads the methods we can use to analyze them, and in particular it permits to apply the functional analytic results explained so far. Moreover, as proved in Proposition 2.2.7 of \cite{ScQA}, the category of sheaves with values on a quasi-abelian category is quasi-abelian and therefore it makes sense to calculate the derived functors of the global section functor using the theory of quasi-abelian categories. 

In particular, the calculation of the cohomology of coherent sheaves, thought as sheaves with values in locally convex of bornological vector spaces, gives more informations about the sheaves because not only the exactness of the maps is detected but also the strictly exactness, as next lemma explains.

\begin{lemma} \label{lemma:quasi}
	Let $\sF$ be a coherent sheaf over an analytic space $X$ such that
	\[ \R \Gamma(X, \sF) = \Gamma(X, \sF) \]
	in $D(\bLoc_k)$ or $D(\bBorn_k)$, then $\R \Gamma(X, \sF) = \Gamma(X, \sF)$ in $D(\bAb)$.
\end{lemma}
\begin{proof}
The underlying complex of abelian groups of a strictly exact complex of locally convex of bornological vector space is an exact complex of abelian groups.
\end{proof}

The converse of Lemma \ref{lemma:quasi} does not need to hold.

\begin{lemma} \label{lemma:thm_B_closed_immersion_strong}
	Let $(X, \cO_X)$ be a analytic space which satisfies Theorem B when coherent sheaves are thought as sheaves with value in $\bLoc_k$ or $\bBorn_k$ and $(Y, \cO_Y) \rhook (X, \cO_X)$ a closed immersion. Then, coherent sheaves over $Y$ have vanishing cohomology in $D(\bLoc_k)$ or $D(\bBorn_k)$ too.
\end{lemma}
\begin{proof}
	The proofs in the first part of section I.1 of \cite{GR} applies also for sheaves with values in $\bLoc_k$ and $\bBorn_k$, giving that the direct image functor associated to a closed immersion is exact.  Again, the direct image of a coherent sheaf with respect to a closed immersion is a coherent sheaf by Theorem 3 \S 16 of \cite{Serre}.
\end{proof}

\subsection{A reinterpretation of the proof of Kiehl's Theorem}

In this section we briefly explain how to (re)deduce the Theorem B of Kiehl using the results presented in section \ref{sec:der_lim}. This proof also serves as blueprint (or maybe as a lemma) for the main proof of this paper, which will be built over the ideas explained in this section. So, for this section $k$ is constrained to be a non-Archimedean valued field. We notice that the argument we give is not very different from the original one of Kiehl, but the way we rewrite it makes easier to spot its essential parts and thus it makes more clear how to generalize the result. 

\begin{thm} \label{thm:kiehl}
	Let $X$ be a quasi-Stein space. Then, $X$ satisfies Theorem B when coherent sheaves of $X$ are thought as sheaves with values in $\bLoc_k$.
\end{thm}
\begin{proof}
   Let $\sF$ be a coherent sheaf on $(X, \cO_X)$. Consider an affinoid exhaustion
	\[ X = \bigcup_{n \in \N} U_n. \]
	It is a standard result of \u{C}ech cohomology that the covering $\{U_n\}_{n \in \N}$ can be used to calculate the \u{C}ech cohomology of $\sF$. For each $n \in \N$ we can consider the inclusion $\iota_n: U_n \to X$ and the coherent sheaf $\sF_n = \iota_n^{*}(\sF)$ on $U_n$ and the push-forward sheaf $(\iota_n)_*(\sF_n)$ on $X$. The value of $(\iota_n)_*(\sF_n)$ can be easily calculated as 
	\[ (\iota_n)_*(\sF_n) (V) = \sF(U_n \cap V) \]
	for each admissible open $V \subset X$. Therefore, as a sheaf
	\[ \sF \cong \limpro_{n \in \N} (\iota_n)_*(\sF_n) \]
	and hence for each admissible open $V \subset X$
	\[ (\limpro_{n \in \N} (\iota_n)_*(\sF_n))(V) \cong \limpro_{n \in \N} ((\iota_n)_*(\sF_n)(V)) \]
	because limits of sheaves are calculated as pre-sheaves. Consider now the \u{C}ech complex for the covering $\{ U_n \}_{n \in \N}$ of $X$,
	\[ 0 \to \sF(X) \to \prod_{i \in \N} \sF(U_i) \to \prod_{i,j \in \N, i < j} \sF(U_i \cap U_j) \to \cdots \]
	this is equivalent to 
	\[ 0 \to (\limpro_{n \in \N} (\iota_n)_*(\sF_n))(X) \to \prod_{i \in \N} (\limpro_{n \in \N} (\iota_n)_*(\sF_n))(U_i) \to \prod_{i,j \in \N, i < j} (\limpro_{n \in \N} (\iota_n)_*(\sF_n))(U_i \cap U_j) \to \cdots \]
	which is equivalent to
	\[ 0 \to \limpro_{n \in \N} \sF_n(U_n) \to \prod_{i \in \N} \limpro_{n \in \N} \sF_n(U_i \cap U_n) \to \prod_{i,j \in \N, i < j} \limpro_{n \in \N} \sF_n(U_i \cap U_j \cap U_n) \to \cdots \]
	which is equivalent to
	\[ 0 \to \limpro_{n \in \N} \sF(U_n) \to \limpro_{n \in \N} \prod_{i \in \N} \sF(U_i \cap U_n) \to \limpro_{n \in \N} \prod_{i,j \in \N, i < j} \sF(U_i \cap U_j \cap U_n) \to \cdots . \]
	For each $n \in \N$ the complex
	\[ 0 \to \sF_n(U_n) \to \prod_{i \in \N} \sF_n(U_i \cap U_n) \to \prod_{i,j \in \N, i < j} \sF_n(U_i \cap U_j \cap U_n) \to \cdots \]
	is strictly exact, because $\sF_n$ is coherent on $U_n$ and $U_n$ satisfies Theorem B, as a consequence of Tate's Theorem. For $m \ge 0$ the projective limit
	\[ \limpro_{n \in \N} \prod_{i_0 < \dots < i_m \in \N} \sF_n(U_{i_0} \cap \dots \cap U_{i_m} \cap U_n) \cong \prod_{i_0 < \dots < i_m \in \N}  \limpro_{n \in \N} \sF_n(U_{i_0} \cap \dots \cap U_{i_m} \cap U_n) \]
	is exact, because each term $\sF_n(U_{i_0} \cap \dots \cap U_{i_m} \cap U_n)$ is eventually constant when $n \to \infty$ (because $U_{i_0} \cap \dots \cap U_{i_m} \cap U_n = U_j$ with $j = \min \{ i_0, \dots, i_m, n \}$) and because direct products are exact in $\bLoc_k$. Therefore, we can apply Lemma \ref{lemma:complexes} to deduce the theorem once we checked that also the projective limit 
\begin{equation} \label{eqn:F}
\limpro_{n \in \N} \sF_n(U_n)
\end{equation} 
has vanishing higher derived functors. This follows as an easy application of the classical Mittag-Leffler Lemma for Fr\'echet spaces \ref{lem:TopML}. The maps $\cO_X(U_{n + 1}) \to \cO_X(U_n)$ are with dense images, because they are Weierstrass localizations. Therefore, for an $r$ such that the upper horizontal maps of the diagram
	\[
\begin{tikzpicture}
\matrix(m)[matrix of math nodes,
row sep=2.6em, column sep=2.8em,
text height=1.5ex, text depth=0.25ex]
{  \cO_{U_n}(U_n)^r  & \sF(U_n) \\
   \cO_{U_n}(U_n)^r \wotimes_{\cO_{U_n}(U_n)} \cong \cO_{U_{n- 1}}(U_{n - 1})^r  & \sF(U_{n-1}) \cong \sF(U_n) \wotimes_{\cO_{U_n}(U_n)} \cO_{U_{n- 1}}(U_{n - 1})  \\};
\path[->,font=\scriptsize]
(m-1-1) edge node[auto] {$$} (m-1-2);
\path[->,font=\scriptsize]
(m-1-1) edge node[auto] {$$} (m-2-1);
\path[->,font=\scriptsize]
(m-1-2) edge node[auto] {$$}  (m-2-2);
\path[->,font=\scriptsize]
(m-2-1) edge node[auto] {$$}  (m-2-2);
\end{tikzpicture}.
\]
is surjective, also the bottom horizontal is surjective because tensor product preserve surjections. The left vertical one is dense, which implies that the right one is dense because tensor products preserve dense maps. Hence, we can apply Lemma \ref{lem:TopML} to deduce that the limit (\ref{eqn:F}) has vanishing higher functors.
\end{proof}

\begin{cor}
Let $X$ be a quasi-Stein space. Then, $X$ satisfies Theorem B in the classical sense.
\end{cor}
\begin{proof}
It is enough to apply Lemma \ref{lemma:quasi} in combination with Theorem \ref{thm:kiehl}.
\end{proof}

\begin{rmk}
We stated our version of Kiehl's result in $\bLoc_k$, and not in $\bBorn_k$ because we do not even know if it is true in $\bBorn_k$! Notice that to apply the Mittag-Leffler Lemma for bornological Fr\'echet spaces one needs to know that the limit space, in our case $\underset{n \in \N}\limpro \sF_n(U_n)$, is nuclear. But this is false for quasi-Stein spaces in general. It is true only for the Stein ones. Notice that thanks to \ref{cor:proj_born_top} we know that also in the bornological case the \u{C}ech complex of $\sF$ is exact (considering only the underlying abelian groups) but we do not know if it is strictly exact when $\sF_n(U_n)$ are equipped with the von Neumann bornologies and then the projective limit bornology is calculated. Nevertheless the next theorem holds.
\end{rmk}

\begin{thm} \label{thm:kiehl_born}
Let $X$ be a quasi-Stein space. Then, $X$ satisfies Theorem B when coherent sheaves of $X$ are thought as sheaves with values in $\bBorn_k$ if the section spaces $\sF(U)$ are equipped with the compactoid bornology induced by their locally convex topologies.
\end{thm}
\begin{proof}
We do not enter in a detailed explanation of what are compactoid subsets of a locally convex space. We refer the reader to Section 3.3 of \cite{BaBeKr} for a review of this notion. We notice that to deduce the theorem it is enough to check that the functor $\bCpt: \bLoc_k \to \bBorn_k$, which associates to a locally convex space the underlying $k$-vector space equipped with the compactoid bornology, is exact when restricted to the category of Fr\'echet spaces. And this is proved in  Lemma 3.66 of \cite{BaBeKr}.
\end{proof}

\begin{open} 
Is the previous theorem true also when coherent sheaves are equipped with the von Neumann bornology?
\end{open}

\subsection{The Theorem B for dagger quasi-Stein spaces}

In this section we remove the restriction on $k$ to be non-Archimedean. 

\begin{lemma} \label{lemma:B_dagger}
	Let $W$ be a nuclear LB-space and $\{V_n\}_{n \in \N}$ be a projective system of nuclear LB spaces such that $\underset{n \in \N}\limpro V_n$ is a nuclear Fr\'echet space and $\underset{n \in \N}{``\limpro"} V_n \cong \underset{n \in \N}{``\limpro"} V_n'$, where $\underset{n \in \N}{``\limpro"} V_n'$ is an epimorphic system of Fr\'echet spaces. Then
	\[ \R \limpro_{n \in \N} W \wotimes_k V_n \cong \limpro_{n \in \N} W \wotimes_k V_n \]
	holds in $D(\bLoc_k)$ and $D(\bBorn_k)$.
\end{lemma}
\begin{proof}
We recall that an LB space is nuclear as locally convex space if and only if it is nuclear as bornological space (cf. Lemma 3.59 of \cite{BaBeKr}).

The Roos complex
\[ 0 \to \limpro_{n \in \N} V_n \to \prod_{n \in \N} V_n \stackrel{\Delta_{V_\bullet}}{\to} \prod_{n \in \N} V_n \to 0 \]
is strictly exact both in $\bLoc_k$ and $\bBorn_k$ because we can apply the Mittag-Leffler Lemmas to the system $\{V_n'\}_{n \in \N}$ and then apply Lemma \ref{lemma:pro_object}. Since $W$ is nuclear, the functor $W \wotimes_k (-)$ is exact (for the exactness in $\bLoc_k$ the result is classical, and see Theorem 3.50 of \cite{BaBeKr} for the bornological version of the result) and therefore the sequence
\[ 0 \to W \wotimes_k \limpro_{n \in \N} V_n \to W \wotimes_k \prod_{n \in \N} V_n \stackrel{id_{W} \wotimes \Delta_{V_\bullet}}{\to} W \wotimes_k \prod_{n \in \N} V_n \to 0 \]
is strictly exact. In $\bLoc_k$ the completed (projective) tensor product commutes with products and cofiltered limits, therefore we get a strictly exact sequence
\[ 0 \to \limpro_{n \in \N} W \wotimes_k V_n \to \prod_{n \in \N} W \wotimes_k V_n \stackrel{\Psi}{\to} \prod_{n \in \N} W \wotimes_k V_n \to 0. \]
To end the proof we need to check that the map $\Psi$ coincides with the map $\Delta_{W \wotimes_k V_\bullet}$ of the Roos complex of the system $\{W \wotimes_k V_n\}_{n \in \N}$. Notice that the isomorphism 
\[ W \wotimes_k \prod_{n \in \N} V_n \to \prod_{n \in \N} W \wotimes_k V_n \]
is given by the extension by continuity (and by linearity) of the map 
\[ a \otimes (b_i)_{i \in \N} \mapsto (a \otimes b_i)_{i \in \N}, \]
see \cite{J}, Theorem 15.4.1 for a proof of this fact. We denote this isomorphism by $\varphi$. Therefore,
\[ \Psi((a \otimes b_i)_{i \in \N}) = \varphi((id_{W} \otimes \Delta_{V_\bullet} )( a \otimes (b_i)_{i \in \N})) = \varphi(a \otimes (b_i - \pi_{i, i+1}(b_{i + 1}))_{i \in \N}) = \]
\[ = (a \otimes b_i - a \otimes \pi_{i, i+1}(b_{i + 1}))_{i \in \N} = \Delta_{W \wotimes_k V_\bullet}((a \otimes b_i)_{i \in \N}). \]
Therefore, the map $\Delta_{W \wotimes_k V_\bullet}$ is surjective. Thus, Lemma \ref{lemma:born_proj} implies that $\Delta_{W \wotimes_k V_\bullet}$ is surjective also as a map of bornological vector spaces. Finally, applying Theorem \ref{thm:web_quotient} and Lemma \ref{thm:born_LB_quotient} we obtain that $\Delta_{W \wotimes_k V_\bullet}$ is a quotient map of both locally convex and bornological vector spaces.
\end{proof}

\begin{cor}
The projective system 
\[ \limpro_{\rho < 1} W_k^1 \wotimes_k W_k^1(\rho) \]
which defines the quasi-Stein algebra associated to the direct product of a closed and open unitary disk, has vanishing higher derived functors.
\end{cor}
\begin{proof}
Lemma \ref{cor:B_dagger} applies directly to the system once one choose a strictly increasing sequence of non negative real numbers $\{ \rho_n \}_{n \in \N}$ whose limit is $1$, to replace the given system with an equivalent one with $\N$ as index set.
\end{proof}

We need a last lemma.

\begin{lemma} \label{lemma:top_rings}
Let $R$ be a complete topological ring and $R' \subset R$ a subring which is dense in $R$. Let $E$ be a complete topological $R$-module and $E' \subset E$ an $R'$-module. Then, $\ol{E'}$ is an $R$-submodule of $E$.
\end{lemma}
\begin{proof}
The set $\ol{E'}$ is by definition the set of limit points of convergent nets of elements of $E'$. Moreover, each element $a \in R$ can be written as a limit of a convergent net of elements of $R'$. Therefore,
given $x = \underset{i \in I}\lim \ x_i \in \ol{E'}$, with $x_i \in E'$, and $a = \underset{j \in J}\lim \ a_j$, with $a_j \in R'$, then
\[ a x = a \lim_{i \in I} x_i = \lim_{i \in I} a x_i \]
because multiplication by $a$ is a continuous operation and
\[  \lim_{i \in I} a x_i = \lim_{i \in I} (\lim_{j \in J} a_j) x_i = \lim_{i \in I} (\lim_{j \in J} a_j x_i) = \lim_{i,j \in I \times J} a_j x_i \in \ol{E'} \]
because $\ol{E'}$ is complete, being a closed submodule of a complete $R$-module.
\end{proof}

\begin{thm} \label{thm:B_dagger}
	Let $X$ be the direct product of a finite number of open polydisks, affine lines and dagger closed polydisks. Then, $X$ satisfies Theorem B (in both cases when coherent sheaves are considered as sheaves with values in $\bLoc_k$ or $\bBorn_k$).
\end{thm}
\begin{proof}
Let $\sF$ be a coherent sheaf on $X$. We can use the same reasoning used to prove Theorem \ref{thm:kiehl} to reduce the question to check that
\[ \sF(X) \cong \limpro_{n \in \N} \sF(U_n) \cong \R \limpro_{n \in \N} \sF(U_n) \]
in $D(\bLoc_k)$ or $D(\bBorn_k)$. Moreover, we notice that it is enough to settle the case when $X$ is the direct product of a closed disk of dimension $1$ with an open disk of dimension $1$ to then deduce the theorem by an easy induction argument. Therefore, we will discuss only this case.

Lemma \ref{lemma:B_dagger} settles immediately the case when $\sF = \cO_X$. Then, the case of $\sF = \cO_X^r$ follows immediately from the previous one, because
\[ \limpro_{n \in \N} \cO_X^r(U_n) \cong (\limpro_{n \in \N}\cO_X(U_n))^r \] 
and products are exact fucntors.

Now, consider a general coherent sheaf $\sF$ on $X$. Since $\sF|_{U_n}$ is coherent on $U_n$ by the dagger version of Kiehl's Theorem on coherent sheaves on affinoid spaces (see the first section of Chapter 6 of \cite{Bam}), $\sF$ is associated to a finitely generated $\cO_{U_n}(U_n)$-module. Therefore there exists a surjective map $\cO_X^r(U_n) \to \sF(U_n)$. We claim that such a surjective map $\cO_X^r|_{U_n} \to \sF|_{U_n}$ can be found to be induced by a map of sheaves $\a: \cO_X^r \to \sF$ (which is not necessarily surjective on global sections). To find this map, it is enough to check that the values of the standard generators of $\cO_{U_n}^r(U_n)$ are mapped in $\pi_n(\sF(X)) \subset \sF(U_n)$ by $\a|_{U_n} = \a_n$. Indeed, by Lemma \ref{lemma:top_rings} the subset $\pi_n(\sF(X))$ can be chosen to be dense in $\sF(U_n)$ (otherwise one can replace $\sF(U_n)$ with $\ol{\pi_n(\sF(X))}$ in $\sF(U_n)$ obtaining an equivalent system to the given one. See also Proposition 2.6.1 of \cite{J} for the issue of reducedness of projective limits of locally convex spaces). Since $\cO_{U_n}(U_n)$ is a Noetherian algebra, $\sF(U_n)$ a finitely generated $\cO_{U_n}(U_n)$-module and all $\cO_{U_n}(U_n)$-submodule of $\sF(U_n)$ are closed (see Chapter 6 of \cite{Bam} for a proof of this fact), we see by a simple induction argument we can find a surjective map $\a_n: \cO_{U_n}^r(U_n) \to \sF(U_n)$ which maps the standard generators of $\cO_{U_n}^r(U_n)$ to elements of $\pi_n(\sF(X))$.

Now, we can apply Theorem \ref{thm:LB_proj_0_strong} (using $m = n + 1$, which is always a possible by re-indexing the system) to the projective system which defines $\cO_X^r(X)$ to deduce that there exists bouned $B_n \subset \cO_{U_n}^r(U_n)$ such that
\[ \pi_{n, n+1}(\cO_{U_{n + 1}}^r(U_{n+1})) \subset \pi_n(\cO_X^r(X)) + B_n. \]
Then, since $\a_n: \cO_{U_n}^r(U_n) \to \sF(U_n)$ is bounded and surjective we deduce that
\[ \pi_{n, n+1}(\sF(U_{n+1})) \subset \a_n(\pi_n(\cO_X^r(X))) + \a_n(B_n) \]
but since $\a_n$ is induced by a map of sheaves $\a_n(\pi_n(\cO_X^r(X))) \subset \pi_n(\sF(X))$, therefore we can deduce that
\[ \pi_{n, n+1}(\sF(U_{n + 1})) \subset \pi_n(\sF(X)) + \a_n(B_n). \]
Therefore, the projective system that defines $\sF$ satisfies the hypothesis of Theorem \ref{thm:LB_proj_0_strong} and since $\sF(U_n)$ are LB spaces we can apply Theorem \ref{thm:web_quotient} and Theorem \ref{thm:born_LB_quotient} to deduce that 
\[ \R \limpro_{n \in \N} \sF(U_n) \cong \limpro_{n \in \N} \sF(U_n) \]
both in $D(\bLoc_k)$ and $D(\bBorn_k)$.
\end{proof}

\begin{cor} \label{cor:B_dagger}
	Let $X$ be dagger quasi-Stein space that can be embedded as a closed subspace in a direct product of a finite number of open polydisks, affine lines and dagger closed polydisks. Then, $X$ satisfies theorem B.
\end{cor}
\begin{proof}
The corollary is proven using Theorem \ref{thm:B_dagger} and Lemma \ref{lemma:thm_B_closed_immersion_strong}. Moreover, using Lemma \ref{lemma:quasi} we can deduce the Theorem B in its classical formulation.
\end{proof}

\begin{rmk}
We want to emphasize the fact that Theorem \ref{thm:B_dagger} and Corollary \ref{cor:B_dagger} apply also when the base field is Archimedean (\ie $\R$ or $\C$) and, in this case, it can be interpreted within the theory developed in \cite{Bam}. For example, they imply the Theorem B for spaces like products of closed and open polydisks over $\C$, which is a result of which we do not know any appearance in the literature about complex analytic spaces and complex geometry.
\end{rmk}

\section{Conclusions}

As first concluding remark, we underline that our study of coherent sheaves as bornological modules over dagger quasi-Stein algebras, beside studying them as locally convex modules, is not motivated by the sake of completeness, but because of our previous works which use the bornological point of view in analytic geometry as more natural. In particular, the category of (complete) bornological vector spaces as the advantage of being a closed symmetric monoidal category with respect to the category of locally convex spaces, which is not. In \cite{BaBe} and \cite{BaBeKr} we explain how to interpret analytic geometry as relative algebraic geometry in the sense of T\"oen-Vezzosi, with respect to the category of complete bornological vector spaces, as a first step to give foundations for a theory of derived analytic geometry. The vanishing of the cohomology of coherent sheaves, when considered as complete bornological modules, means that the natural cohomology theory in the framework described in \cite{BaBe} and \cite{BaBeKr} gives the expected results. Therefore, this work si also a consistency check for that proposed theory.

The Theorem B proved so far can be used to simplify computations of rigid cohomology and computations related to the overconvergent de Rham-Witt cohomology in the following way. In Theorem 1.4 of \cite{Bert} and Proposition 4.37 of \cite{DLZ} the authors were interested to show that the cohomology theories they are discussing are independent of the compactifications and of the formal models of the compactifications chosen. In both cases the problem is reduced to check that the higher cohomology groups of sheaves of differential forms vanishes on spaces of the form $A \times B$, where $A$ is a dagger affinoid space and $B$ is an open polydisk or similarly to the case when $A$ is a Stein space and $B$ is a closed polydisk with its dagger affinoid structure. See the proofs of claim (ii) and (iii) of Theorem 1.4 of \cite{Bert} for more details about these computations, and the last part of section 4 of \cite{DLZ}. Our Theorem \ref{thm:B_dagger} has these computations as particular cases and implies them from a more theoretical viewpoint.

We conclude this work discussing the following conjecture. 

\begin{conj} \label{conj:embedding}
Let $X$ be a (dagger) quasi-Stein space of finite dimension. Then, there exist $n_s, n_a \in \N$ and a proper injective embedding
\[ X \rhook \A_k^{n_s} \times \B_k^{n_a} \]  
where $\B_k$ is the closed unit ball of $k$, with its affinoid structure if $X$ is quasi-Stein and with its dagger affinoid structure if $X$ is dagger quasi-Stein. 
\end{conj}

This conjecture implies that the Theorem B we proved so far extends to all dagger quasi-Stein spaces. Indeed, we can apply Theorem 6.1.9 of \cite{Poi} to the proper map $X \rhook \A_k^{n_s} \times \B_k^{n_a}$. The use of Theorem 6.1.9 of \cite{Poi} assume the Direct Image Theorem for dagger analytic spaces, which is proved in \cite{GK} only under restrictive hypothesis. A general enough version of the Direct Image Theorem can be proved that allows to deduce the general Theorem B from Conjecture \ref{conj:embedding}, but it is probably possible to avoid the use of the Direct Image Theorem by studying the actual properties of the conjectured embeddings (which are of a very particular form). We do not enter here in a further discussion of this problem, which will be discussed in future works.

We notice that Conjecture \ref{conj:embedding} would generalize the following classical results:

\begin{thm} (Embedding Theorem) \\
	Let $X$ be Stein space of finite dimension. Then, there exist an $n \in \N$ and a proper injective embedding
	\[ X \rhook \A_k^n. \]
\end{thm}
\begin{proof}
See Theorem 4.25 of \cite{Lut} for the non-Archimedean case of the theorem and the Embedding Theorem at page 126 of \cite{GR} for the complex analytic case of the theorem.
\end{proof}

Notice that Conjecture \ref{conj:embedding} would not imply that all quasi-Stein spaces are of the form $A \times S$, where $A$ is an affinoid space and $S$ a Stein space, as next example show.

\begin{exa}
	 Consider the Zariski closed subspace of $\B_k^+ \times \B_k^-$, where $\B_k^+$ is the closed disk of radius $1$ and $\B_k^-$ is the open one, given by the equation $X_1 X_2 = \frac{1}{2}$ (where $X_1$ and $X_2$ are the coordinates of the disks). This quasi-Stein space is the half-open and half-closed disk of radius $1$ and $\frac{1}{2}$ and it can be realized with the outer border or with the inner border, depending if it is projected to $\B_k^+$ or to $\B_k^-$.
\end{exa}

\end{document}